\documentclass{article}

\usepackage{amscd,amsmath,amssymb,amsthm}
\newcommand{\eps}{\varepsilon}

\newcommand{\R}{\mathbb R}
\newcommand{\N}{\mathbb N}

\newcommand{\then}{\Longrightarrow}
\newcommand{\J}{{\cal J}}
\newcommand{\A}{{\cal A}}

\newcommand{\M}{{\cal M}}
\newcommand{\enne}{{\cal N}}

\DeclareMathOperator{\codim}{codim}
\DeclareMathOperator*{\esssup}{ess\; sup}

\newtheorem{corollary}{Corollary}[section]
\newtheorem{theorem}[corollary]{Theorem}
\newtheorem{lemma}[corollary]{Lemma}
\newtheorem{proposition}[corollary]{Proposition}

\theoremstyle{definition}
\newtheorem{definition}[corollary]{Definition}
\newtheorem{remark}[corollary]{Remark}

\numberwithin{equation}{section}
\begin{document}

\title{{\bf Multiple solutions for some\\
 symmetric supercritical problems}\thanks{The 
research that led to the present paper was partially supported by  
Fondi di Ricerca di Ateneo {\sl ``Metodi variazionali e topologici nello studio di fenomeni non
lineari''} and Research Funds INdAM -- GNAMPA Project 2017 
``Metodi variazionali per fenomeni non--locali''}}

\author{A.M. Candela, G. Palmieri, A. Salvatore\\
{\small Dipartimento di Matematica}\\
{\small Universit\`a degli Studi di Bari Aldo Moro} \\
{\small Via E. Orabona 4, 70125 Bari, Italy}\\
{\small \it annamaria.candela@uniba.it}\\
{\small \it giuliana.palmieri.uniba@gmail.com}\\ 
{\small \it addolorata.salvatore@uniba.it}}
\date{}

\maketitle

\begin{abstract}
The aim of this paper is investigating the existence of one or more 
critical points of a family of functionals which generalizes the model problem 
\[
\bar J(u)\ =\ \frac1p\ \int_\Omega \bar A(x,u)|\nabla u|^p dx - \int_\Omega G(x,u) dx
\]
in the Banach space $X = W^{1,p}_0(\Omega)\cap L^\infty(\Omega)$, 
where $\Omega \subset \R^N$ is an open bounded domain, $1 < p < N$ and the real terms
$\bar A(x,t)$ and $G(x,t)$ are $C^1$ Carath\'eo\-do\-ry functions on $\Omega \times \R$.

We prove that, even if the coefficient $\bar A(x,t)$ makes
the variational approach more difficult, if it satisfies ``good'' growth assumptions then at least one
critical point exists also when the nonlinear term $G(x,t)$ has 
a suitable supercritical growth. Moreover, if the functional is even,
it has infinitely many critical levels.    

The proof, which exploits the interaction between two different norms on $X$, 
is based on a weak version of the Cerami--Palais--Smale condition and a suitable intersection lemma
which allow us to use a Mountain Pass Theorem.
\end{abstract}



\section{Introduction} \label{secintroduction}

Here, we look for critical points of the nonlinear functional
\[
\J(u)\ =\ \int_\Omega A(x,u,\nabla u) dx - \int_\Omega G(x,u) dx,\qquad
u \in {\cal D}\subset W^{1,p}_0(\Omega),
\]
which generalizes the model problem 
\begin{equation}\label{modello}
\bar J(u) = \frac1p \int_\Omega \bar A(x,u) |\nabla u|^p dx - \int_\Omega G(x,u) dx,\qquad
u \in {\cal D}\subset W^{1,p}_0(\Omega),
\end{equation}
where $\Omega$ is an open bounded domain in $\R^N$, $1 < p < N$, 
$A :\Omega \times \R \times \R^N \to \R$, respectively $\bar A :\Omega \times \R \to \R$,
and $G : \Omega \times \R \to \R$ are given functions.

We note that, even in the simplest case $A(x,u,\nabla u)= \frac1p \bar A(x,t)|\nabla u|^p$
and $G(x,t)\equiv 0$, with $\bar A(x,t)$ smooth, bounded 
away from zero but $\frac{\partial \bar A}{\partial t}(x,t) \not\equiv 0$, 
the functional $\bar J$ is defined in $W^{1,p}_0(\Omega)$ but  
is G\^ateaux differentiable only along directions of 
$W^{1,p}_0(\Omega) \cap L^\infty(\Omega)$. 

In the past, such a problem has been overcome by introducing suitable definitions
of critical point for $\J$ and related existence results have been stated (see,
e.g., \cite{AB1,AB2,Ca,PeSq}). Here, as in \cite{CP2}, 
suitable assumptions assure that the functional
$\J$ is $C^1$ in $X=W^{1,p}_0(\Omega) \cap L^\infty(\Omega)$ (see Proposition \ref{smooth1})
and its Euler--Lagrange equation is
\begin{equation}\label{euler}
\left\{
\begin{array}{ll}
- {\rm div} (a(x,u,\nabla u)) + A_t(x,u,\nabla u) = g(x,u)
  &\hbox{in $\Omega$,}\\
u\ = \ 0 & \hbox{on $\partial\Omega$,}
\end{array}
\right.
\end{equation}
where 
\begin{equation}\label{grad}
\hbox{$A_t(x,t,\xi) = \frac{\partial A}{\partial t}(x,t,\xi),\;
a(x,t,\xi) = (\frac{\partial A}{\partial \xi_1}(x,t,\xi),\dots,\frac{\partial A}{\partial \xi_N}(x,t,\xi))$,}\;
G(x,t) = \int_0^t g(x,s) ds.
\end{equation}

We note that, from a physical point of view, problem \eqref{euler} is interesting for its applications. 
For example, if $\Omega = \R^N$ and $A(x,u,\nabla u) = (1 + |u|^2) |\nabla u|^2$,   
model equations of \eqref{euler} appear in Mathematical Physics 
and describe several physical phenomena in the theory of superfluid film and in dissipative
quantum mechanics (for more details, see \cite{LWW} and references therein).

In order to find solutions of \eqref{euler}, i.e. critical points of $\J$ in $X$,
we cannot apply directly existence and 
multiplicity results similar to the classical Ambrosetti--Rabinowitz theorems (see \cite{AR,BBF}).
Indeed, our functional $\J$ does not satisfy the Palais--Smale condition in $X$ as it has Palais--Smale sequences 
which converge in $W^{1,p}_0(\Omega)$ but are unbounded in $L^\infty(\Omega)$ 
(see, e.g., \cite[Example 4.3]{CP2017}). Hence, we have to weaken the definition of Palais--Smale condition 
(see Definition \ref{wCPS}) and use it for stating a generalized version of the Mountain Pass Theorems (see Theorems
\ref{mountainpass} and \ref{abstract}).

In \cite{CP2} the existence of critical points of the functional $\J$, i.e. solutions of \eqref{euler}, has been already 
proved if $p>1$, $A(x,t,\xi)$ satisfies suitable assumptions and $G(x,t)$ has a 
$p$--superlinear growth which has to be subcritical if $p<N$. 
Anyway, even if the dependence from $t$ of the 
principal part $A(x,t,\xi)$ makes
the variational approach more difficult, 
it can allow the nonlinear term $G(x,t)$ to be supercritical.
In fact, the aim of this paper is to extend 
the main statements in \cite{CP2} to a function $G(x,t)$ with critical or supercritical growth
if  $1<p<N$: roughly speaking, we prove that the more $A(x,t,\xi)$ is unbounded and 
grows with respect to $t$, the more $G(x,t)$ can have a supercritical growth.

Since our main theorems need a list of hypotheses, we give 
their complete statements in Section \ref{sec5} (see Theorems \ref{mainesisto} and \ref{mainmolti}),
anyway, here, in order to highlight how our approach improves 
previous results, we consider the particular setting
\[
A(x,t,\xi) \ =\ \frac1p (A_1(x) + A_2(x) |t|^{ps}) |\xi|^p \quad \hbox{and} \quad
g(x,t)\ =\ |t|^{\mu-2}t
\]
with $1 < p < N$, $s \ge 0$, $\mu \ge 1$, so that problem \eqref{euler} reduces to
\begin{equation}\label{euler0}
\left\{
\begin{array}{ll}
- {\rm div} ((A_1(x) + A_2(x) |u|^{ps}) |\nabla u|^{p-2}\nabla u) + s A_2(x)|u|^{ps-2} u |\nabla u|^p\
 =\ |u|^{\mu-2}u
  &\hbox{in $\Omega$,}\\
u\ = \ 0 & \hbox{on $\partial\Omega$.}
\end{array}
\right.
\end{equation}

If $s=0$, then problem \eqref{euler0} has been widely studied
(see, e.g., \cite{PAO} and references therein). On the contrary,
if $s > 0$ we obtain the following result.

\begin{theorem}\label{mod0}
Let $A_1$, $A_2 \in L^{\infty}(\Omega)$ be two given functions such that 
\begin{equation}\label{su0}
A_1(x) \ge \alpha_0,\qquad A_2(x) \ge \alpha_0\qquad \hbox{for a.e. $x \in \Omega$,}
\end{equation}
for a constant $\alpha_0 > 0$. Assume that
\begin{equation}\label{caso0}
2 < 1 + p < p(s+1) < \mu < p^*(s+1),
\end{equation}
where $p^*$ is the critical exponent.
Then, problem \eqref{euler0} has infinitely many weak bounded solutions.
\end{theorem}

To our knowledge, there are very few results dealing with quasilinear supercritical problems.
Usually, they make use of a suitable change of variables which reduces the supercritical
problem to a subcritical one (see, e.g., \cite{LWW}).
Unluckily, such an approach works only if $A(x,t,\xi)$ has a very particular form,
and so, for example, it is not allowed also in the simplest case $A_2(x)=1$ but $A_1(x)$ not constant. 
Different arguments can be found in \cite{ABO} where, by using a sequence of
truncated functionals, the authors 
prove that problem \eqref{euler0} with, e.g., $p=2$, has at least
one positive solution if \eqref{caso0} and the further condition
$2 (s+1) < 2^*$ hold, which imply $N < 6$ (see \cite[Theorem 2.1]{ABO}). 
Differently from \cite{ABO}, here we use variational methods which 
exploit the interaction between two different norms 
and we do not require this additional restriction
(see also \cite{CS2018} where, in the same setting of Theorem \ref{mod0},
the existence of at least one positive solution of problem \eqref{euler0} is 
proved). 
\smallskip

This paper is organized as follows.
In Section \ref{abstractsection} we introduce the weak Cerami--Palais--Smale condition
and prove some related abstract existence and multiplicity results which generalize 
the Mountain Pass Theorem (see \cite[Theorem 2.2]{Ra1}) and its symmetric version (see \cite[Theorem 9.12]{Ra1}).
In Section \ref{variational}, after introducing the hypotheses for $A(x,t,\xi)$ and $G(x,t)$, 
we give the variational formulation of our problem and prove that $\J$ satisfies the weak Cerami--Palais--Smale condition. 
Finally, in Section \ref{sec5} the main results are stated and proved.  


\section{Abstract setting}
\label{abstractsection}

We denote $\N = \{1, 2, \dots\}$ and, throughout this section, we assume that:
\begin{itemize}
\item $(X, \|\cdot\|_X)$ is a Banach space with dual 
$(X',\|\cdot\|_{X'})$,
\item $(W,\|\cdot\|_W)$ is a Banach space such that
$X \hookrightarrow W$ continuously, i.e. $X \subset W$ and a constant $\sigma_0 > 0$ exists
such that
\begin{equation}
\label{continuity}
\|u\|_W \ \le \ \sigma_0\ \|u\|_X\qquad \hbox{for all $u \in X$,}
\end{equation}
\item $J : {\cal D} \subset W \to \R$ and $J \in C^1(X,\R)$ with $X \subset {\cal D}$,
\item $K_J = \{u \in X:\ dJ(u) = 0\}$ is
the set of the critical points of $J$ in $X$.
\end{itemize}
Furthermore, fixing $\beta \in \R$, we define
\begin{itemize}
\item $K_J^\beta = \{u \in X:\ J(u) = \beta,\ dJ(u) = 0\}$
the set of the critical points of $J$ in $X$ at level $\beta$,
\item $J^\beta = \{u\in X:\ J(u) \le \beta\}$
the sublevel of $J$ with respect to $\beta$,
\end{itemize}
and, taking $r > 0$, by pointing out 
the two different norms $\|\cdot\|_W$ and $\|\cdot\|_X$, we set
\begin{itemize}
\item $B^X_r\ =\ \{u \in X:\ \|u\|_X\ <\ r\},\qquad B^W_r\ =\ \{u \in X:\ \|u\|_W\ <\ r\}$,
\item $\bar{B}^X_r\ =\ \{u \in X:\ \|u\|_X\ \le\ r\}, \qquad \bar{B}^W_r\ =\ \{u \in X:\ \|u\|_W\ \le\ r\}$,
\item $S^X_r\ =\ \{u \in X:\ \|u\|_X\ =\ r\}, \qquad S^W_r\ =\ \{u \in X:\ \|u\|_W\ =\ r\}$.
\end{itemize}

Anyway, in order to avoid any
ambiguity and simplify, when possible, the notation, 
from now on by $X$ we denote the space equipped with
its given norm $\|\cdot\|_X$ while, if a different norm is involved,
we write it explicitly.

For simplicity, taking $\beta \in \R$, we say that a sequence
$(u_n)_n\subset X$ is a {\sl Cerami--Palais--Smale sequence at level $\beta$},
briefly {\sl $(CPS)_\beta$--sequence}, if
\[
\lim_{n \to +\infty}J(u_n) = \beta\quad\mbox{and}\quad 
\lim_{n \to +\infty}\|dJ(u_n)\|_{X'} (1 + \|u_n\|_X) = 0.
\]
Moreover, $\beta$ is a {\sl Cerami--Palais--Smale level}, briefly {\sl $(CPS)$--level}, 
if there exists a $(CPS)_\beta$--sequence.

As $(CPS)_\beta$--sequences may exist which are unbounded in $\|\cdot\|_X$
but converge with respect to $\|\cdot\|_W$,
we have to weaken the classical Cerami--Palais--Smale 
condition in a suitable way according to the ideas already developed in 
previous papers (see, e.g., \cite{CP1,CP2,CP3}).  

\begin{definition} \label{wCPS}
The functional $J$ satisfies the
{\slshape weak Cerami--Palais--Smale 
condition at level $\beta$} ($\beta \in \R$), 
briefly {\sl $(wCPS)_\beta$ condition}, if for every $(CPS)_\beta$--sequence $(u_n)_n$,
a point $u \in X$ exists, such that 
\begin{description}{}{}
\item[{\sl (i)}] $\displaystyle 
\lim_{n \to+\infty} \|u_n - u\|_W = 0\quad$ (up to subsequences),
\item[{\sl (ii)}] $J(u) = \beta$, $\; dJ(u) = 0$.
\end{description}
If $J$ satisfies the $(wCPS)_\beta$ condition at each level $\beta \in I$, $I$ real interval, 
we say that $J$ satisfies the $(wCPS)$ condition in $I$.
\end{definition}

Since in \cite{CP3} a Deformation Lemma has been proved 
if the functional $J$ satisfies a weaker version of the $(wCPS)_\beta$ condition,
namely any $(CPS)$--level is also a critical level, in particular we can state the following result. 

\begin{lemma}[Deformation Lemma] \label{def}
Let $J\in C^1(X,\R)$ and consider $\beta \in \R$ such that
\begin{itemize}
\item $J$ satisfies the $(wCPS)_\beta$ condition, 
\item $\ K_J^\beta = \emptyset$.
\end{itemize}
Then, fixing any $\bar \eps > 0$, there exist
a constant $\eps > 0$ and a homeomorphism $\psi : X \to X$
such that $2\eps < {\bar \eps}$ and
\begin{itemize}
\item[$(i)$] $\psi(J^{\beta+\eps}) \subset J^{\beta-\eps}$,
\item[$(ii)$] $\psi(u) = u$ for all $u \in X$ such that either
$J(u) \le \beta-\bar\eps$ or $J(u) \ge \beta+\bar\eps$.
\end{itemize}
Moreover, if $J$ is even on $X$, then $\psi$ can be chosen odd.
\end{lemma}

\begin{proof}
It is enough to reason as in \cite[Lemma 2.3]{CP3} with $\beta_1 = \beta_2 = \beta$
and to note that the deformation $\psi: X \to X$ is a homeomorphism.    
\end{proof}

From Lemma \ref{def} we obtain the following generalization of the Mountain Pass Theorem
(compare it with \cite[Theorem 1.7]{CP3} and the classical statement in \cite[Theorem 2.2]{Ra1}).

\begin{theorem}
\label{mountainpass}
Let $J\in C^1(X,\R)$ be such that $J(0) = 0$
and the $(wCPS)$ condition holds in $\R_+$.\\
Moreover, assume that there exist a continuous map
$\ell : X \to \R$, some constants $r_0$, $\varrho_0 > 0$, and  
 $e \in X$ such that 
\begin{itemize}
\item[$(i)$] $\; \ell(0) =0 \qquad\hbox{and}\qquad \ell(u) \ge \|u\|_W \quad \hbox{for all $u \in X$}$;
\item[$(ii)$] $\; u \in X, \quad \ell(u) = r_0\qquad \then\qquad J(u) \ge \varrho_0$;
\item[$(iii)$] $\; \|e\|_W > r_0\qquad\hbox{and}\qquad J(e) < \varrho_0$.
\end{itemize}
Then, $J$ has a Mountain Pass critical point $u_0 \in X$ such that $J(u_0) \ge \varrho_0$.
\end{theorem}

Furthermore, with the stronger assumption that $J$ is symmetric, 
 the following generalization of the symmetric Mountain Pass Theorem can be stated
(see \cite[Theorem 1.8]{CP3} and compare with \cite[Theorem 9.12]{Ra1} and 
\cite[Theorem 2.4]{BBF}).

\begin{theorem}
\label{abstract}
Let $J\in C^1(X,\R)$ be an even functional such that $J(0) = 0$ and
the $(wCPS)$ condition holds in $\R_+$.
Moreover, assume that $\varrho > 0$ exists so that:
\begin{itemize}
\item[$({\cal H}_{\varrho})$]
three closed subsets $V_\varrho$, $Z_\varrho$ and $\M_\varrho$ of $X$ and a constant
$R_\varrho > 0$ exist which satisfy the following conditions:
\begin{itemize}
\item[$(i)$] $V_\varrho$ and $Z_\varrho$ are subspaces of $X$ such that
\[
V_\varrho + Z_\varrho = X,\qquad \codim Z_\varrho\ <\ \dim V_\varrho\ <\ +\infty;
\]
\item[$(ii)$] $\M_\varrho = \partial \enne$, where $\enne \subset X$ is a neighborhood of the origin
which is symmetric and bounded with respect to $\|\cdot\|_W$; 
\item[$(iii)$]  $\ u \in \M_\varrho \cap Z_\varrho\qquad \then\qquad J(u) \ge \varrho$;
\item[$(iv)$] $\ u \in V_\varrho, \quad \|u\|_X \ge R_\varrho \qquad \then\qquad J(u) \le 0$.
\end{itemize}
\end{itemize}
Then, if we put 
\[
\beta_\varrho\ =\ \inf_{\gamma \in \Gamma_\varrho} \sup_{u\in V_\varrho} J(\gamma(u)),
\]
with 
\[
\Gamma_\varrho\ =\ \{\gamma : X \to X:\
\gamma\ \hbox{odd homeomeorphism,}\quad \gamma(u) = u \ \hbox{if $u \in V_\varrho$ with $\|u\|_X \ge R_\varrho$}\},
\]
the functional $J$ possesses at least a pair of symmetric critical points in $X$ 
with corresponding critical level $\beta_\varrho$ which belongs to $[\varrho,\varrho_1]$,
where $\varrho_1 \ge \displaystyle \sup_{u \in V_\varrho}J(u) > \varrho$.
\end{theorem}

\begin{remark}
Since in Theorem \ref{abstract} the vector space $V_\varrho$ is finite dimensional,
then condition $({\cal H}_{\varrho})(iv)$ implies that $\displaystyle \sup_{u \in V_\varrho}J(u) < +\infty$,
furthermore it still holds if we replace $\|\cdot\|_X$ with $\|\cdot\|_W$.
\end{remark}

If we can apply infinitely many times Theorem \ref{abstract}, then the following multiplicity abstract 
result can be stated.

\begin{corollary}
\label{multiple}
Let $J \in C^1(X,\R)$ be an even functional such that $J(0) = 0$,
the $(wCPS)$ condition holds in $\R_+$ and 
assumption $({\cal H}_{\varrho})$ holds for all $\varrho > 0$.\\
Then, the functional $J$ possesses a sequence of critical points $(u_n)_n \subset X$ such that
$J(u_n) \nearrow +\infty$ as $n \nearrow +\infty$.
\end{corollary}

The proof of Theorem \ref{abstract} is obtained reasoning as in \cite[Theorem 1.8]{CP3}
by using Lemma \ref{def} and the following result.

\begin{lemma}[Intersection Lemma]\label{int}
Let $V$, $Z$ and $\M$ be closed subsets of $X$ which satisfy conditions
$(i)$ and $(ii)$ in Theorem \ref{abstract}. 
Fixing $R > 0$ and defining
\[
\Gamma_R = \{\gamma : X \to X:\
\gamma\ \hbox{odd homeomeorphism,}\quad \gamma(u) = u \ \hbox{if $u \in V$ with $\|u\|_X \ge R$}\},
\]
then
\[
\gamma(V) \cap \M \cap Z \ne \emptyset\qquad \hbox{for all $\gamma \in \Gamma_R$.}
\]
\end{lemma}

\begin{proof}
Fixing any $\gamma \in \Gamma_R$, for simplicity we denote $Q = \gamma(V) \cap \M \cap Z$.
It is enough to prove that
\begin{equation}\label{kra2}
i_2(Q) \ge \dim V - \codim Z \ge 1,
\end{equation}
where $i_2(\cdot)$ is the Krasnoselskii genus (see, e.g., \cite[Section II.5]{Str}).\\
In order to prove \eqref{kra2}, firstly let us point out that hypotheses $(i)$ and $(ii)$
imply that $Q$ is symmetric with respect to the origin
but $0 \not\in Q$. Moreover, $Q$ is compact in $X$.
In fact, we have $V = (V \cap \bar{B}_R^X) \cup (V\setminus B_R^X)$,
with $V \cap \bar{B}_R^X$ compact (as $\dim V < +\infty$) and
$\gamma(V\setminus B_R^X) = V\setminus B_R^X$ (by the definition of $\Gamma_R$).
Hence, $Q = (\gamma(V \cap \bar{B}_R^X) \cap \M \cap Z)
\cup ((V \setminus B_R^X) \cap \M \cap Z)$ is compact because
$\gamma(V \cap \bar{B}_R^X) \cap \M \cap Z$ is compact (as closed subset of the compact set $\gamma(V \cap \bar{B}_R^X)$)
and $(V \setminus B_R^X) \cap \M \cap Z$ is compact, too, as closed and bounded in the finite dimensional space $V$
(since $\M$ is bounded in $\|\cdot\|_W$ but in $V$ the norms $\|\cdot\|_X$
and $\|\cdot\|_W$ are equivalent).\\
Then, by the continuity, monotonicity and subadditivity properties of the genus, 
an open neighborhood $U$ of $Q$ in $X$ exists such that 
\begin{equation}\label{kra3}
i_2(Q) = i_2(\bar{U}) \ge i_2(\gamma(V) \cap \M \cap \bar{U}) \ge 
i_2(\gamma(V) \cap \M) - i_2(\gamma(V) \cap (\M \setminus U)).
\end{equation}
Now, denoting by $V^*$ the complement of $Z$, from hypothesis $(i)$ it follows that
$V^* \subset V$; furthermore, it has to be $\gamma(V) \cap (\M \setminus U) \subset V^* \setminus \{0\}$,
hence 
\begin{equation}\label{kra4}
i_2(\gamma(V) \cap (\M \setminus U)) \le \dim V^* = \codim Z.
\end{equation}
On the other hand, since $\gamma$ is an odd homeomorphism on $X$, assumption $(ii)$ implies that
the set $V \cap \gamma^{-1}(\M)$ is the boundary of a bounded symmetric neighborhood of the origin in $V$.
Then, from \cite[Proposition 5.2]{Str} we have
\[
i_2(\gamma(V) \cap \M) = i_2(V \cap \gamma^{-1}(\M)) =  \dim V, 
\]  
which, together with \eqref{kra3} and \eqref{kra4}, implies \eqref{kra2}.
\end{proof}


\section{Variational setting and first properties}
\label{variational}

From now on, let $\Omega \subset \R^N$ be an open bounded domain, $N\ge 2$,
so we denote by:
\begin{itemize}
\item $L^q(\Omega)$ the Lebesgue space with
norm $|u|_q = \left(\int_\Omega|u|^q dx\right)^{1/q}$ if $1 \le q < +\infty$;
\item $L^\infty(\Omega)$ the space of Lebesgue--measurable 
and essentially bounded functions $u :\Omega \to \R$ with norm
\[
|u|_{\infty} = \esssup_{\Omega} |u|;
\]
\item $W^{1,p}_0(\Omega)$ the classical Sobolev space with
norm $\|u\|_{W} = |\nabla u|_p$ if $1 \le p < +\infty$;
\item $|C|$ the usual Lebesgue measure of a measurable set $C$ in $\R^N$.
\end{itemize}

From now on, let $\, A : \Omega \times \R \times \R^N \to \R\,$
and $\, g :\Omega \times \R \to \R\,$ be such that, considering the notation in \eqref{grad},
the following conditions hold:
\begin{itemize}
\item[$(H_0)$]
$A(x,t,\xi)$ is a $C^1$ Carath\'eodory function, i.e., \\
$A(\cdot,t,\xi) : x \in \Omega \mapsto A(x,t,\xi) \in \R$ is measurable for all $(t,\xi) \in \R\times \R^N$,\\
$A(x,\cdot,\cdot) : (t,\xi) \in \R\times \R^N \mapsto A(x,t,\xi) \in \R$ 
is $C^1$ for a.e. $x \in \Omega$;
\item[$(H_1)$] a real number $p > 1$ and 
some positive continuous 
functions $\Phi_i$, $\phi_i : \R \to \R$, $i \in \{1,2\}$, exist such that
\[
\begin{array}{ccll}
|A_t(x,t,\xi)| &\le& \Phi_1(t) + \phi_1(t)\ |\xi|^p 
& \hbox{a.e. in $\Omega$, for all $(t,\xi) \in \R\times \R^N$,}\\
|a(x,t,\xi)| &\le& \Phi_2(t) + \phi_2(t)\ |\xi|^{p-1}
& \hbox{a.e. in $\Omega$, for all $(t,\xi) \in \R\times \R^N$;}
\end{array}
\]
\item[$(G_0)$] 
$g(x,t)$ is a Carath\'eodory function, i.e.,\\
$g(\cdot,t) : x \in \Omega \mapsto g(x,t) \in \R$ is measurable for all $t \in \R$;\\
$g(x,\cdot) : t \in \R \mapsto g(x,t) \in \R$ is continuous for a.e. $x \in \Omega$;
\item[$(G_1)$] $a_1$, $a_2 > 0$ and $q \ge 1$ exist such that
\[
|g(x,t)| \le a_1 + a_2 |t|^{q-1} \qquad
\hbox{a.e. in $\Omega$, for all $t \in \R$.}
\]
\end{itemize}

\begin{remark}
From $(G_1)$ it follows that there exist $a_3$, $a_4 > 0$ such that
\begin{equation}
\label{alto3}
|G(x,t)| \le a_3 + a_4 |t|^q\qquad\hbox{a.e in $\Omega$, for all $t \in \R$.}
\end{equation}
We note that, unlike assumption $(G_1)$ in \cite{CP2}, no upper bound on $q$ is actually required. 
\end{remark}

In order to investigate the existence of weak solutions  
of the nonlinear problem \eqref{euler}, the notation introduced for the abstract 
setting at the beginning of Section \ref{abstractsection}
is referred to our problem with $W = W^{1,p}_0(\Omega)$ and the
Banach space $(X,\|\cdot\|_X)$ defined as
\begin{equation}\label{space}
X := W^{1,p}_0(\Omega) \cap L^\infty(\Omega),\qquad
\|u\|_X = \|u\|_W + |u|_\infty
\end{equation}
(here and in the following, $|\cdot|$ denotes 
the standard norm on any Euclidean space as the dimension
of the considered vector is clear and no ambiguity arises).\\
Moreover, from the Sobolev Embedding Theorem, for any $r \in [1,p^*[$,
$p^* = \frac{pN}{N-p}$ as $N > p$,
a constant $\sigma_r > 0$ exists, such that 
\[
|u|_r\ \le\ \sigma_r \|u\|_W \quad \hbox{for all $u \in W^{1,p}_0(\Omega)$}
\]
and the embedding $W^{1,p}_0(\Omega) \hookrightarrow\hookrightarrow L^r(\Omega)$
is compact.

From the definition of $X$, we have that $X \hookrightarrow W^{1,p}_0(\Omega)$ and $X \hookrightarrow L^\infty(\Omega)$
with continuous embeddings, and \eqref{continuity} holds with $\sigma_0 = 1$. 
If $p > N$ then $X = W^{1,p}_0(\Omega)$, as $W^{1,p}_0(\Omega) \hookrightarrow L^\infty(\Omega)$; hence, classical 
Mountain Pass Theorems in \cite{AR} can be used.

Now, we consider the functional $\J : X \to \R$ defined as
\begin{equation}
\label{funct}
\J(u)\ =\ \int_\Omega A(x,u,\nabla u) dx - \int_\Omega G(x,u) dx,\qquad
u \in X.
\end{equation}

Taking any $u$, $v\in X$, by direct computations
it follows that its G\^ateaux differential in $u$ along the direction $v$ is
\begin{equation}
\label{diff}
\langle d\J(u),v\rangle = \int_\Omega (a(x,u,\nabla u)\cdot \nabla v + A_t(x,u,\nabla u)v) dx
- \int_\Omega g(x,u)v dx .
\end{equation}

The following proposition extends \cite[Proposition 3.1]{CP2} in which the regularity 
of $\J$ is stated only if $G(x,t)$ has a subcritical growth. 

\begin{proposition}\label{smooth1}
Let us assume that 
conditions $(H_0)$--$(H_1)$, $(G_0)$--$(G_1)$ hold and 
two positive continuous functions $\Phi_0$, $\phi_0 : \R \to \R$ exist such that 
\begin{equation}\label{uno}
|A(x,t,\xi)| \le \Phi_0(t) + \phi_0(t)\ |\xi|^p\quad\hbox{a.e. in $\Omega$,
for all $(t,\xi) \in \R\times \R^N$.}
\end{equation}
If $(u_n)_n \subset X$, $u \in X$ are such that
\begin{equation}\label{succ1}
\|u_n - u\|_W \to 0, \; u_n \to u\; \hbox{a.e. in $\Omega$} \ \quad\hbox{if $n \to+\infty$}
\end{equation}
\begin{equation}\label{succ2}
\hbox{and $M > 0$ exists so that $|u_n|_\infty \le M$ for all $n \in \N$,}
\end{equation}
then
\[
\J(u_n) \to \J(u)\quad \hbox{and}\quad \|d\J(u_n) - d\J(u)\|_{X'} \to 0
\quad\hbox{if $\ n\to+\infty$.}
\]
Hence, $\J$ is a $C^1$ functional on $X$ with Fr\'echet differential  
defined as in \eqref{diff}.
\end{proposition}

\begin{proof}
As in the first part of the proof of \cite[Proposition 3.1]{CP2},
from assumptions $(H_0)$--$(H_1)$ and \eqref{uno} the functional
\[
\A : u \in X \ \mapsto\ \A(u) = \int_\Omega A(x,u,\nabla u) dx \in \R 
\]
is such that
$\A(u_n) \to \A(u)$ and $\|d\A(u_n) - d\A(u)\|_{X'} \to 0$,
with 
\[
\langle d\A(u),v\rangle\ =\ \int_\Omega a(x,u,\nabla u) \cdot \nabla v\ dx + 
\int_\Omega A_t(x,u,\nabla u) v dx, \quad
u, v \in X.
\]
On the other hand, from $(G_0)$ and \eqref{succ1} it follows that  
$G(x,u_n) \to G(x,u)$ and $g(x,u_n) \to g(x,u)$ a.e. in $\Omega$, then 
$(G_1)$, \eqref{alto3}, \eqref{succ2} and Lebesgue's Dominated Convergence Theorem imply
that also the functional
\[
{\cal G}: u \in X \ \mapsto\ {\cal G}(u) = \int_\Omega G(x,u) dx \in \R
\]
is such that ${\cal G}(u_n) \to {\cal G}(u)$ and $\|d{\cal G}(u_n) - d{\cal G}(u)\|_{X'} \to 0$,
with 
\[
\langle d{\cal G}(u),v\rangle \ =\ \int_\Omega g(x,u)v dx \quad \hbox{for all $u$, $v \in X$.}
\]
Then, the conclusion follows.
\end{proof}

In order to prove more properties of the functional $\J$ in \eqref{funct},
we require that some constants $\alpha_i > 0$,
$i \in \{1,2,3\}$, $\eta_j > 0$, $j\in\{1,2\}$, and $s \ge 0$, $\mu > p$, $R_0 \ge 1$, 
exist such that the following hypotheses are satisfied:
\begin{itemize}
\item[$(H_2)$] $\ A(x,t,\xi) \le \eta_1 a(x,t,\xi)\cdot \xi\quad$ a.e. in $\Omega\; $ if $|(t,\xi)| \ge R_0$;
\item[$(H_3)$] $\ |A(x,t,\xi)| \le \eta_2\quad$ a.e. in $\Omega\; $ if $|(t,\xi)| \le R_0$;
\item[$(H_4)$] $\ a(x,t,\xi)\cdot \xi \ge \alpha_1 (1 + |t|^{p s})|\xi|^p \quad$ 
a.e. in $\Omega$, $\;$ for all $(t,\xi) \in \R\times \R^N$;
\item[$(H_5)$] $\ a(x,t,\xi)\cdot \xi + A_t(x,t,\xi) t \ge \alpha_2 a(x,t,\xi)\cdot \xi\quad$
a.e. in $\Omega\;$ if $|(t,\xi)| \ge R_0$;
\item[$(H_6)$] $\ \mu A(x,t,\xi) - a(x,t,\xi)\cdot \xi - A_t(x,t,\xi) t \ge \alpha_3 a(x,t,\xi)\cdot \xi
\quad$ a.e. in $\Omega\;$ if $|(t,\xi)| \ge R_0$;
\item[$(H_7)$] $\ $ for all $\xi$, $\xi^* \in \R^N$, $\xi \ne \xi^*$, it is
\[
 [a(x,t,\xi) - a(x,t,\xi^*)]\cdot [\xi - \xi^*] > 0 \quad \hbox{a.e. in $\Omega$, for all $t\in \R$;}
\]
\item[$(G_2)$] $\ g(x,t)$ satisfies the Ambrosetti--Rabinowitz condition, i.e.
\[
 0 < \mu G(x,t) \le g(x,t) t\quad \hbox{for a.e. $x \in \Omega$ if $|t| \ge R_0$.}
\]
\end{itemize}

\begin{remark}\label{per1}
If in $(H_5)$ we take $t=0$ and $|\xi| \ge R_0$, we deduce that
$\alpha_2 \le 1$.\\ 
Moreover, from hypotheses $(H_5)$ and $(H_6)$ it follows that
\begin{equation}\label{tt5}
\mu A(x,t,\xi)\ \ge\ (\alpha_2 + \alpha_3)\ a(x,t,\xi)\cdot \xi
\quad\hbox{a.e. in $\Omega$ if $|(t,\xi)| \ge R_0$;}
\end{equation}
hence, if also $(H_4)$ holds, for a.e. $x \in\Omega$ we have that  
\begin{equation}\label{tt}
A(x,t,\xi)\ \ge\ \alpha_1\ \frac{\alpha_2 + \alpha_3}{\mu}\ (1 + |t|^{p s})\ |\xi|^p \ge 0
\quad\hbox{if $|(t,\xi)| \ge R_0$.}
\end{equation}
Thus, from \eqref{tt} and $(H_3)$, for a.e. $x \in\Omega$ we obtain that 
\begin{equation}\label{ttbis}
A(x,t,\xi)\ \ge\ \alpha_1\ \frac{\alpha_2 + \alpha_3}{\mu}\ (1 + |t|^{p s})\ |\xi|^p - \eta_3
\quad\hbox{for all $(t,\xi) \in \R\times\R^N$}
\end{equation}
for a suitable $\eta_3 > 0$.
\end{remark}

\begin{remark}\label{smooth}
From $(H_1)$--$(H_6)$, since \eqref{tt} is verified, then 
\begin{equation}\label{tt2t}
|A(x,t,\xi)|\ \le\ \eta_1\ (\Phi_2(t) + \phi_2(t))\ |\xi|^p\ + \ \eta_1 \Phi_2(t) + \eta_2 
\end{equation}
a.e. in $\Omega$, for all $(t,\xi) \in \R\times\R^N$. 
Whence, the growth condition \eqref{uno} holds and Proposition \ref{smooth1} applies.
\end{remark}

\begin{remark}\label{per2}
With respect to estimate \eqref{tt2t}, more precise growth conditions 
on $A(x,t,\xi)$ can be deduced.
In fact, taken $|(t,\xi)| \ge R_0$, hypotheses $(H_2)$ and $(H_6)$
imply
\[
\mu A(x,t,\xi)\ \ge\ \frac{1+\alpha_3}{\eta_1} A(x,t,\xi) + A_t(x,t,\xi) t\quad\hbox{a.e. in $\Omega$.}
\]
Hence, we have
\begin{equation}\label{tt1}
(\mu - \frac{1+\alpha_3}{\eta_1}) A(x,t,\xi) \ge A_t(x,t,\xi) t
\quad\hbox{a.e. in $\Omega$ if $|(t,\xi)| \ge R_0$,}
\end{equation}
where, without loss of generality, just taking $\eta_1$ large enough,
we can always have 
\[
\mu > \frac{1+\alpha_3}{\eta_1}.
\]
Thus, by means of \eqref{tt2t}, \eqref{tt} and \eqref{tt1},
direct calculations allow one to prove the existence of a constant $\eta_4 > 0$
so that
\begin{equation}\label{tt7}
A(x,t,\xi)\ \le\ \eta_4\ |t|^{\mu - \frac{1+\alpha_3}{\eta_1}}\ |\xi|^p
\quad\hbox{a.e. in $\Omega$, if $|t| \ge 1$ and $|\xi| \ge R_0$.}
\end{equation}
Whence, \eqref{tt5} and \eqref{tt7} imply
\begin{equation}\label{tt6}
a(x,t,\xi)\cdot\xi \le\ \frac{\eta_4\mu}{\alpha_2 + \alpha_3}\ |t|^{\mu - \frac{1+\alpha_3}{\eta_1}}\ |\xi|^p
\quad\hbox{a.e. in $\Omega$, if $|t| \ge 1$ and $|\xi| \ge R_0$.}
\end{equation}
At last, $(H_4)$ and \eqref{tt6} imply that
\begin{equation}\label{good1}
0 \le p s \le \mu - \frac{1+\alpha_3}{\eta_1}.
\end{equation}
We note that, if 
\begin{equation}\label{good2}
0 \le s < \frac{\mu}{p},
\end{equation}
then, without loss of generality, we can always choose $\eta_1$ in $(H_2)$ large enough so that 
\eqref{good1} holds. 
\end{remark}

\begin{remark}\label{come0}
In the model case $A(x,t,\xi) = \frac1p \bar A(x,t) |\xi|^p$ conditions $(H_2)$ and 
$(H_7)$ are trivially verified, so the set of assumptions reduce to the following one:
\begin{itemize}
\item[$(H_0)'$] $\ \bar A(x,t)\; $ is a $C^1$ Carath\'eodory function in $\Omega \times \R$;
\item[$(H_1)'$] two positive continuous functions $\Phi_i : \R \to \R$, $i \in \{1,2\}$, exist such that
\[
|\bar A_t(x,t)| \le \Phi_1(t), \quad |\bar A(x,t)| \le \Phi_2(t)\qquad 
\hbox{a.e. in $\Omega$, for all $t \in \R$;}
\]
\item[$(H_4)'$] $\ \bar A(x,t) \ge \alpha_1 (1 + |t|^{p s})\quad$ 
a.e. in $\Omega$, $\;$ for all $t \in \R$;
\item[$(H_5)'$] $\ \bar A(x,t) + \frac1p \bar A_t(x,t) t \ge \alpha_2 \bar A(x,t)\quad$
a.e. in $\Omega\;$ if $|t| \ge R_0$;
\item[$(H_6)'$] $\ \big(\frac{\mu}p - 1\big) \bar A(x,t) - \frac1p \bar A_t(x,t) t \ge \alpha_3 \bar A(x,t)
\quad$ a.e. in $\Omega\;$ if $|t| \ge R_0$.
\end{itemize}
In particular, if we consider $\bar A(x,t) = A_1(x) + A_2(x) |t|^{ps}$ as in \eqref{euler0},
the previous hypotheses hold if $A_1$, $A_2 \in L^{\infty}(\Omega)$ are such that \eqref{su0}
is satisfied and 
\begin{equation}\label{su00}
2 < 1 + p < p(s+1) < \mu. 
\end{equation}
\end{remark}

\begin{remark} 
Conditions $(G_0)$ and $(G_2)$ imply that a function $\eta \in L^\infty(\Omega)$,
$\eta(x) > 0$ a.e. in $\Omega$, and a constant $a_5 \ge 0$ exist such that
\begin{equation}
\label{basso3}
G(x,t)\ \ge\ \eta(x)\ |t|^\mu - a_5\qquad\hbox{a.e. in $\Omega$, for all $t \in \R$.}
\end{equation}
Hence, if also $(G_1)$ holds, from \eqref{alto3}, \eqref{good2} and \eqref{basso3} it follows
\[
p s < \mu \le q.
\]
\end{remark}

If the assumptions in this section hold with $s = 0$ in $(H_4)$ 
and $q < p^*$ in $(G_1)$, from \cite[Proposition 4.6]{CP2}
it follows that the functional $\J$ in \eqref{funct} satisfies the $(wCPS)$ condition in $\R$.
Here, in order to extend such a result to the case $s > 0$, and then considering  
$G(x,t)$ with a critical or supercritical growth,
we need the following application of the Rellich Embedding Theorem.

\begin{lemma}\label{rellich}
Taking $1 < p < N$ and $s \ge 0$, let $(u_n)_n \subset X$ be a sequence such that
\begin{equation}\label{c1}
\left( \int_\Omega (1+|u_n|^{p s})\ |\nabla u_n|^p dx\right)_n\quad \hbox{is bounded.} 
\end{equation}
Then, $u \in W^{1,p}_0(\Omega)$ exists such that 
$|u|^s u \in W^{1,p}_0(\Omega)$, too, and, up to subsequences, if $n\to+\infty$ we have
\begin{eqnarray}
&&u_n \rightharpoonup u\ \hbox{weakly in $W^{1,p}_0(\Omega)$,}
\label{c2}\\
&&|u_n|^s u_n \rightharpoonup |u|^s u\ \hbox{weakly in $W^{1,p}_0(\Omega)$,}
\label{c7}\\
&&u_n \to u\ \hbox{a.e. in $\Omega$,}
\label{c3}\\
&&u_n \to u\ \hbox{strongly in $L^r(\Omega)$ for each $r \in [1,p^*(s+1)[$.}
\label{c4}
\end{eqnarray}
\end{lemma}

\begin{proof}
Firstly, we note that
\begin{equation}\label{c8}
|\nabla (|u|^{s}u)|^p \ = \
(s+1)^p\ |u|^{ps}\ |\nabla u|^p \quad \hbox{a.e. in $\Omega\;$ for all $u \in X$,}
\end{equation}
then from \eqref{c1} the sequences $(u_n)_n$ and $(|u_n|^{s}u_n)_n$ 
are bounded in $W^{1,p}_0(\Omega)$; hence,
$u$, $v \in W^{1,p}_0(\Omega)$ exist such that,
up to subsequences, we have \eqref{c2}, \eqref{c3},\eqref{c4} with $r<p^*$,
and also
$|u_n|^s u_n \rightharpoonup v\ $ weakly in $W^{1,p}_0(\Omega)$ and 
$|u_n|^{s}u_n \to v\ $ a.e. in $\Omega$. Thus, $v = |u|^s u$
and \eqref{c7} holds.\\
At last, if $s > 0$, \eqref{c4} holds also if $p^* \le r < p^*(s+1)$
from interpolation as $u \in L^{p^*(s+1)}(\Omega)$
and $(u_n)_n$ is bounded in $L^{p^*(s+1)}(\Omega)$.
\end{proof}

Now, we recall a particular version of \cite[Theorem II.5.1]{LU} which 
we will use for proving the boundedness 
of the weak limit of a $(CPS)$--sequence (see \cite[Lemma 4.5]{CP2}).

\begin{lemma}
\label{tecnico} Let $p$, $r$ be so that $1 < p \le r < p^*$, $p < N$
and take $v \in W^{1,p}_0(\Omega)$. Assume that $\bar{a} >0$ and $k_0\in \N$
exist such that the inequality
\[
\int_{\Omega^+_k}|\nabla v|^p dx \le \bar{a}\left(|\Omega^+_k| +
\int_{\Omega^+_k} v^r dx\right)
\]
holds for all $k \ge k_0$, with $\Omega^+_k = \{x \in \Omega: v(x) > k\}$.
Then, $\displaystyle \esssup_{\Omega} v$
is bounded from above by a positive constant which can be chosen so
that it depends only on $|\Omega|$, $N$, $p$, $r$, $\bar{a}$, $k_0$, $|v|_{p^*}$. 
\end{lemma}

Now, we are ready to prove that $\J$ satisfies the weak Cerami--Palais--Smale 
condition in $X$. If $1<p<N$, this new result extends \cite[Proposition 4.6]{CP2} 
where the exponent $q$ in $(G_1)$ is subcritical, i.e., $q < p^*$.
On the contrary, here we assume the weaker condition 
\begin{equation}\label{cps0}
q < p^*(s+1). 
\end{equation}
Hence, without loss of generality, we can always assume $q$ large enough
such that
\begin{equation}\label{ccps}
p (s+1) \ <\ q\ < \ p^*(s+1).
\end{equation}

\begin{proposition}\label{wcps}
Assume that hypotheses $(H_0)$--$(H_7)$, $(G_0)$--$(G_2)$ and \eqref{cps0} hold with $1<p<N$.
Then, the functional $\J$ satisfies the $(wCPS)$ condition in $\R$.
\end{proposition}

\begin{proof}
Let $\beta \in \R$ be fixed and consider a $(CPS)_\beta$--sequence $(u_n)_n \subset X$,
i.e., 
\begin{equation}\label{cps1}
\J(u_n) \to \beta \quad \hbox{and}\quad \|d\J(u_n)\|_{X'}(1 + \|u_n\|_X) \to 0.
\end{equation}
We divide our proof in the following  steps:
\begin{itemize}
\item[1.] $(u_n)_n$ is bounded in $W^{1,p}_0(\Omega)$, or more precisely \eqref{c1} holds; 
thus from Lemma \ref{rellich} a function $u \in W^{1,p}_0(\Omega)$ 
exists such that $|u|^s u \in W^{1,p}_0(\Omega)$
and \eqref{c2}--\eqref{c4} hold, up to subsequences;
\item[2.] $u \in L^\infty(\Omega)$;
\item[3.] 
if $k \ge \max\{|u|_\infty, R_0\} + 1$ ($R_0 \ge 1$ as in the set of hypotheses)
then
\[ 
\J(T_ku_n) \to \beta \quad \hbox{and}\quad 
\|d\J(T_ku_n)\|_{X'} \to 0,
\]
where $T_k : \R \to \R$ is the truncation function defined as
\[
T_kt = \left\{\begin{array}{ll}
t&\hbox{if $|t| \le k$}\\
k\frac t{|t|}&\hbox{if $|t| > k$}
\end{array}\right. ;
\]
\item[4.] $\|T_k u_n - u\|_W \to 0$ if $n\to+\infty$, then
$\|u_n - u\|_W \to 0$ if $n\to+\infty$, too;
\item[5.] $\J(u) = \beta$ and $d\J(u) = 0$.
\end{itemize}
For simplicity, here and in the following we will use the notation $(\eps_n)_n$
for any infinitesimal sequence depending only on $(u_n)_n$ 
while $d_i$ will denote any strictly positive constant independent of $n$.\\
{\sl Step 1.} From \eqref{funct}, \eqref{diff}, \eqref{cps1}, together with
$(H_1)$, $(H_3)$, $(H_6)$, \eqref{alto3}, $(G_1)$, $(G_2)$, by
reasoning as in the proof of Step 1 in \cite[Proposition 4.6]{CP2}
and using hypothesis $(H_4)$ we have that
\[
\begin{split}
\mu \beta + \eps_n\ &=\ \mu \J(u_n) - \langle d\J(u_n),u_n\rangle\
\ge\ \alpha_3 \int_\Omega a(x,u_n,\nabla u_n)\cdot \nabla u_n dx -d_1\\
&\ge\ \alpha_1 \alpha_3 \int_\Omega (1+|u_n|^{ps})\ |\nabla u_n|^p dx - d_1
\end{split}
\] 
which implies \eqref{c1}.\\
{\sl Step 2.} Arguing by contradiction, let us assume that 
\begin{equation}
\label{ess1}
\esssup_\Omega u = +\infty;
\end{equation}
thus, taking any $k \in\N$, $k> R_0$ ($R_0 \ge 1$ as in the hypotheses), we have that
\begin{equation}\label{asp}
|\Omega^+_k| > 0 \quad\hbox{with}\quad \Omega^+_k = \{x \in \Omega: u(x) > k\}.
\end{equation}
Now, for any $\tilde{k} > 0$ consider the new function
$R^+_{\tilde{k}} : t \in\R \to R^+_{\tilde{k}}t \in \R$ such that
\[
R^+_{\tilde{k}}t = \left\{\begin{array}{ll}
0&\hbox{if $t \le \tilde{k}$}\\
t-{\tilde{k}}&\hbox{if $t > \tilde{k}$}
\end{array}\right. .
\]
Taking $\tilde{k} = k^{s+1}$, from \eqref{c7} it follows that 
\[
R_{k^{s+1}}^+(|u_n|^{s}u_n) \rightharpoonup R_{k^{s+1}}^+(|u|^{s}u)\quad
\hbox{weakly in $W^{1,p}_0(\Omega)$;}
\]
then, the weak lower semicontinuity of $\|\cdot\|_{W}$
implies
\[
\int_{\Omega} |\nabla R_{k^{s+1}}^+(|u|^{s}u)|^p dx \ \le\
\liminf_{n\to+\infty}\int_{\Omega} |\nabla R_{k^{s+1}}^+(|u_n|^{s}u_n)|^p dx ,
\]
i.e.,
\begin{equation}\label{b1}
\int_{\Omega^+_k} |\nabla (u^{s+1})|^p dx \le \liminf_{n\to+\infty}\int_{\Omega^+_{n,k}} |\nabla (u_n^{s+1})|^p dx
\end{equation}
as $\; |t|^s t > k^{s+1}\; \iff \; t > k$,
with $\Omega^+_{n,k} = \{x \in \Omega: u_n(x) > k\}$.\\
On the other hand, from $\|R^+_ku_n\|_X \le \|u_n\|_X$, \eqref{cps1} and
\eqref{asp} it follows that $n_{k}\in \N$ exists so that
\begin{equation}\label{b2}
|\langle d\J(u_n),R^+_ku_n\rangle|\ <\ |\Omega^+_k| \qquad \hbox{for all $n \ge n_{k}$.}
\end{equation}
From \eqref{diff}, $(H_5)$ with $\alpha_2 \le 1$ (see Remark \ref{per1}),
$(H_4)$, \eqref{c8}, we have that
\[\begin{split}
\langle d\J(u_n),R^+_ku_n\rangle &=
\int_{\Omega^+_{n,k}} (1 - \frac k{u_n}) \left(a(x,u_n,\nabla u_n)\cdot \nabla u_n 
+ A_t(x,u_n,\nabla u_n)u_n\right) dx \\
&\qquad\; + \int_{\Omega^+_{n,k}} \frac{k}{u_n}\ a(x,u_n,\nabla u_n) \cdot\nabla u_n dx -
\int_{\Omega} g(x,u_n)R^+_ku_n dx \\
& \ge\
\alpha_2\ \int_{\Omega^+_{n,k}} a(x,u_n,\nabla u_n) \cdot \nabla u_n dx -
\int_{\Omega} g(x,u_n)R^+_ku_n dx \\
& \ge\ \alpha_1 \alpha_2 \int_{\Omega^+_{n,k}} u_n^{ps}\ |\nabla u_n|^p dx
-\int_{\Omega} g(x,u_n)R^+_ku_n dx \\
&=\ \frac{\alpha_1 \alpha_2}{(s+1)^p} \int_{\Omega^+_{n,k}} |\nabla (u_n^{s+1})|^p dx
-\int_{\Omega} g(x,u_n)R^+_ku_n dx .
\end{split}
\]
Thus, from \eqref{b2} it follows that
\[
\int_{\Omega^+_{n,k}} |\nabla (u_n^{s+1})|^p dx\ \le\ \frac{(s+1)^p}{\alpha_1 \alpha_2}\
\left(|\Omega^+_k| + \int_{\Omega} g(x,u_n)R^+_ku_n dx\right).
\]
Now, from $(G_1)$, \eqref{c4} and \eqref{cps0} it results 
\[
\int_{\Omega} g(x,u_n)R^+_ku_n dx\ \to\ \int_{\Omega} g(x,u)R^+_ku\ dx;
\]
hence, by passing to the lower limit, \eqref{b1} implies 
\[
\int_{\Omega^+_k} |\nabla (u^{s+1})|^p dx \le \frac{(s+1)^p}{\alpha_1 \alpha_2}\
\left(|\Omega^+_k| + \int_{\Omega} g(x,u)R^+_ku\ dx \right).
\]
Therefore, as in $\Omega^+_k$ it is $u > 1$, from $(G_1)$ and direct computations
it follows that
\begin{equation}\label{b3}
\int_{\Omega^+_k} |\nabla (u^{s+1})|^p dx\ \le\ d_2\
\left(|\Omega^+_k| + \int_{\Omega^+_k} u^q\ dx \right).
\end{equation}
At last, if we set $v = |u|^{s}u$, as $v \in W^{1,p}_0(\Omega)$ and
$\Omega^+_k = \{x \in \Omega: v(x) > k^{s+1}\}$ (in particular, $v = u^{s+1}$ in $\Omega^+_k$),
from \eqref{b3} we obtain
\[
\int_{\Omega^+_k} |\nabla v|^p dx\ \le\ d_2\
\left(|\Omega^+_k| + \int_{\Omega^+_k} v^{\frac{q}{s+1}}\ dx \right).
\]
Then, from \eqref{ccps} Lemma \ref{tecnico} applies and $\displaystyle \esssup_\Omega v < +\infty$ in contradiction to
\eqref{ess1}. 
Similar arguments apply if $\displaystyle\esssup_\Omega (-u) = +\infty$. Hence, $u \in L^\infty(\Omega)$.\\
{\sl Step 3}. The proof can be obtained reasoning as in the proof of Step 3 in
\cite[Proposition 4.6]{CP2} but using \eqref{c4} and \eqref{cps0} instead of \cite[(4.15)]{CP2}.\\
{\sl Steps 4}, {\sl 5}. The proofs are as in the corresponding steps of the proof of
\cite[Proposition 4.6]{CP2}.  
\end{proof}

At last, in order to prove a multiplicity result, 
we introduce a suitable decomposition of $X$. 

If $p=2$, we deal with the Hilbert space $H^1_0(\Omega)$ 
so the classical choice is to consider the sequence 
of the eigenvalues of $-\Delta$ on $\Omega$, 
with homogeneous Dirichlet data, and their (bounded) eigenfunctions, 
so that, for each $n\ge 1$, the Banach space $X$ 
can be decomposed into the
closed subspace spanned by the first $n$ of such eigenfunctions and
the corresponding complement (for the model problem in this case, see
\cite{CP1}).

More in general, if $p > 1$ and $p \ne 2$,  
$W^{1,p}_0(\Omega)$ is just a reflexive Banach space and a ``canonical'' 
decomposition is not known. Anyway, as in \cite[Section 5]{CP2},
a sequence of positive numbers $(\lambda_j)_j$ exists such that
\begin{itemize}
\item
$ 0 < \lambda_1 < \lambda_2 \le \dots \le \lambda_j\le \dots \quad$ and
$\quad \lambda_j \nearrow +\infty \quad $ as $j \to +\infty$;
\item
for each $j \in \N$
a function $\varphi_j\in W^{1,p}_0(\Omega)$ exists such that
$|\varphi_{j}|_p = 1$, $\|\varphi_{j}\|_W = \lambda_{j}$
and $\varphi_i \ne \varphi_j$ if $i \ne j$;
\item $\lambda_1 > 0$ is the first eigenvalue of $-\Delta_p$ 
in $W^{1,p}_0(\Omega)$ such that
\begin{equation}\label{el1}
\lambda_{1}\ \int_{\Omega} |w|^p dx \le\ \int_{\Omega} |\nabla w|^p dx\qquad
\hbox{for all $w\in W^{1,p}_0(\Omega)$}
\end{equation}
and $\varphi_1 \in W^{1,p}_0(\Omega)$ 
is the unique corresponding eigenfunction  
such that $\varphi_1 > 0$, $|\varphi_{1}|_p = 1$ and $\|\varphi_1\|_W = \lambda_1$
(see, e.g., \cite{Lin});
\item $ \varphi_j \in L^{\infty}(\Omega)$ for each $j \in \N$;
\item the sequence $(\varphi_j)_j$ generates the whole 
space $W^{1,p}_0(\Omega)$.
\end{itemize}

Moreover, fixing any $n \in \N$ and defining 
\[
V_n = {\rm span}\{\varphi_1,\dots,\varphi_n\} = \{v \in W^{1,p}_0(\Omega):
\; \exists\ \beta_1,\dots,\beta_n \in \R\ \hbox{s.t.}\
v = \sum_{i=1}^n \beta_i\varphi_i\},
\]
a closed subspace $W_n$ exists such that  
\[
W^{1,p}_0(\Omega)\ =\ V_n + W_n, \qquad  V_n \cap W_n\ = \ \{0\},
\]
and 
\begin{equation}\label{el}
\lambda_{n+1}\ \int_{\Omega} |w|^p dx \le\ \int_{\Omega} |\nabla w|^p dx\qquad
\hbox{for all $w\in W_n$.}
\end{equation}

Then, $V_n$ is a closed subspace of $X$, too, and
 we have that
\begin{equation}\label{decompongo11}
X\ =\ V_n + W_n^X\quad \hbox{and}\quad V_n \cap W_n^X = \{0\},\quad
 \hbox{with $W_n^X = W_n \cap L^\infty(\Omega)$,}
\end{equation}
whence,
\begin{equation}\label{decompongo2}
{\rm codim} W_n^X \ =\ {\rm dim} V_n \ =\ n.
\end{equation}

\section{Existence and multiplicity results}
\label{sec5}

Finally, we can state our main theorems.

\begin{theorem}\label{mainesisto}
Assume that $(H_0)$--$(H_7)$, $(G_0)$--$(G_2)$ 
and \eqref{cps0} hold.
If, furthermore, $\alpha_4 > 0$ exists
such that 
\begin{itemize}
\item[$(H_8)$] $\quad A(x,t,\xi) \ge \alpha_4 (1+|t|^{ps}) |\xi|^p\quad$ a.e. in $\Omega$, 
for all $(t,\xi) \in \R\times \R^N$; 
\item[$(G_3)$] $\quad \displaystyle\limsup_{t\to 0}\frac{g(x,t)}{|t|^{p-2}t}\ 
<\ p \alpha_4 \lambda_1\;$ uniformly with respect to a.e. $x \in \Omega$,
where $\lambda_1$ is the first eigenvalue of $-\Delta_p$ in $W^{1,p}_0(\Omega)$, 
\end{itemize}
then the functional $\J$ defined in \eqref{funct} possesses at least one nontrivial critical point,
i.e., problem \eqref{euler} admits at least a weak bounded nontrivial solution.
\end{theorem}

\begin{remark}
We note that the estimate in hypothesis $(H_8)$
follows from $(H_4)$--$(H_6)$ if $|(t,\xi)| \ge R_0$ 
(see inequality \eqref{tt}). Here, we need such an estimate 
also for $|(t,\xi)| \le R_0$.
\end{remark}

\begin{theorem}\label{mainmolti}
Assume that $(H_0)$--$(H_7)$, $(G_0)$--$(G_2)$ and \eqref{cps0} hold.
Moreover, if $A(x,\cdot,\cdot)$ is even and
$g(x,\cdot)$ is odd for a.e. $x \in \Omega$, then functional $\J$
in \eqref{funct} possesses a sequence of critical points $(u_n)_n \subset X$ 
such that $\J(u_n) \nearrow +\infty$, i.e., 
problem \eqref{euler} admits infinitely many weak bounded solutions.
\end{theorem}

We note that if $s =0$ in $(H_4)$ and $(H_8)$ or if $p \ge N$, then 
Theorems \ref{mainesisto} and \ref{mainmolti} have already been proved in \cite{CP2}. So, here we consider 
$s > 0$ and $1 < p < N$.

Firstly, we define 
\begin{equation}\label{qq0}
\ell_{W,s}(u) \ =\ \max\{\|u\|_W, \| |u|^su\|_W\}\quad \hbox{for all $u \in X$.}
\end{equation}

\begin{remark}\label{continuo}
From \eqref{space} it follows that the map $u \mapsto \| |u|^su\|_W$ 
is well--defined and continuous in $(X,\|\cdot\|_X)$; thus,  
also $\ell_{W,s}:X \to \R$ is continuous with respect to $\|\cdot\|_X$. 
\end{remark}

From now on, assume that $(H_0)$--$(H_7)$, $(G_0)$--$(G_2)$ and \eqref{cps0} hold.

In order to prove that $\J$ satisfies some suitable geometric conditions, 
we need the following lemmas.

\begin{proposition}\label{geo1}
Fixing any $\varrho \in \R$ there exist $n \in \N$, $n=n(\varrho)$, and $r_n > 0$
such that
\begin{equation}\label{qq}
u \in W_n^X, \quad \ell_{W,s}(u) = r_n\qquad \then\qquad \J(u) \ge \varrho,
\end{equation}
where the subspace $W_n^X$ is as in \eqref{decompongo11}.
\end{proposition}

\begin{proof}
Firstly, taking any $u \in X$ we note that \eqref{alto3}
and \eqref{ttbis} imply
\begin{equation}\label{qq1}
\J(u)\ \ge\ \alpha_1\ \frac{\alpha_2+\alpha_3}{\mu}\
\int_{\Omega} (1 + |u|^{ps}) |\nabla u|^{p} dx - a_4 \int_{\Omega} |u|^{q} dx - (\eta_3 + a_3) |\Omega|,
\end{equation}
while from \eqref{c8} and \eqref{qq0} it follows that
\begin{equation}\label{qq11}
\int_{\Omega} (1 + |u|^{ps}) |\nabla u|^{p} dx\ \ge\ \frac{1}{(s+1)^p}\ \big[\ell_{W,s}(u)\big]^p .
\end{equation}
On the other hand, from \eqref{ccps}, a constant
$r > 0$ exists such that $\frac{r}{p}+\frac{q-r}{p^*(s+1)} = 1$,
so, classical interpolation arguments apply and we have
\begin{equation}\label{qq2}
|u|_q^q \le |u|_{p^*(s+1)}^{q-r}\ |u|_p^r,
\end{equation}
where, by the Sobolev Embedding Theorem and \eqref{qq0}, one has 
\begin{equation}\label{qq3}
|u|_{p^*(s+1)}^{q-r}\ =\ ||u|^su|_{p^*}^{\frac{q-r}{s+1}}\ \le\ 
c_* \||u|^su\|_{W}^{\frac{q-r}{s+1}}\ \le\ c_* \big[\ell_{W,s}(u)\big]^{\frac{q-r}{s+1}},
\end{equation}
for a suitable constant $c^* > 0$. \\
Now, fixing any $n \in \N$, from \eqref{el}, \eqref{qq0}, \eqref{qq2} and \eqref{qq3} it follows that
\begin{equation}\label{qq4}
|u|_q^q\ \le\ \ c_* \lambda_{n+1}^{-\frac{r}{p}}\ \|u\|_{W}^{r}\ \big[\ell_{W,s}(u)\big]^{\frac{q-r}{s+1}}\ \le\
c_* \lambda_{n+1}^{-\frac{r}{p}}\ \big[\ell_{W,s}(u)\big]^{r + \frac{q-r}{s+1}}
\quad \hbox{for all $u \in W_n^X$,}
\end{equation}
where from \eqref{ccps} we have 
\[
r + \frac{q-r}{s+1}\ =\ \frac{q+rs}{s+1} > p.
\]
Hence, \eqref{qq1}, \eqref{qq11} and \eqref{qq4} imply that
\[
\begin{split}
\J(u)\ &\ge\ b_1\ \big[\ell_{W,s}(u)\big]^p\ -\ b_2 \lambda_{n+1}^{-\frac{r}{p}}\ 
\big[\ell_{W,s}(u)\big]^{\frac{q+rs}{s+1}} - b_3\\
&=\ b_1 \big[\ell_{W,s}(u)\big]^p\ \left(1\  -\ \frac{b_2}{b_1} 
\lambda_{n+1}^{-\frac{r}{p}}\ \big[\ell_{W,s}(u)\big]^{\frac{q+rs}{s+1} - p}\right) - b_3
\qquad \hbox{for all $u \in W_n^X$,}
\end{split}
\]
for suitable positive constants $b_1$, $b_2$, $b_3$ independent of $n$.\\
Finally, we choose $r_n > 0$ so that
\[
1 -\ \frac{b_2}{b_1}\ \lambda_{n+1}^{-\frac{r}{p}}\ r_n^{\frac{q+rs}{s+1} - p}\ =\ \frac12, 
\]
i.e., 
\begin{equation}\label{qq6}
r_n\ =\ \left(\frac{b_1}{2b_2} \ \lambda_{n+1}^{\frac{r}{p}}\right)^{\frac{s+1}{q - p (s+1) + rs}}.
\end{equation}
Thus, as $\lambda_n \nearrow +\infty$, \eqref{ccps} and \eqref{qq6} imply that
$r_n \nearrow +\infty$, then from the estimate 
\[
\J(u)\ \ge\ \frac{b_1}{2}\ r_n^p\ -\ b_3\qquad \hbox{for all $u \in W_n^X$ with $\ell_{W,s}(u)= r_n$,}
\]
the thesis follows.
\end{proof}

At last, as in \cite[Proposition 6.6]{CP2}, the following statement holds. 

\begin{proposition}\label{geo2}
For any finite dimensional subspace $V$ of $X$, there exists
$R > 0$ such that
\[
\J(u) \le 0\qquad \hbox{for all $u \in V$ with $\|u\|_X \ge R$.}
\]
Hence, $\J$ is bounded from above in $V$.
\end{proposition}
\medskip

\begin{proof}[Proof of Theorem \ref{mainesisto}.]
From $(G_3)$, we can take $\bar\lambda \in \R$ so that
\begin{equation}\label{lamb}
\limsup_{t\to 0}\frac{g(x,t)}{|t|^{p-2}t} < \bar\lambda < p \alpha_4 \lambda_1 .
\end{equation}
Then, from $(G_1)$, \eqref{lamb}  and standard computations, a suitable constant $b_1 > 0$ exists such that 
\[
G(x,t) \le \frac{\bar\lambda}p |t|^p + b_1 |t|^{q}\qquad\hbox{a.e. in $\Omega$, for all $t \in \R$;}
\]
hence, from $(H_8)$, \eqref{c8} and \eqref{el1}, we obtain
\[
\J(u)\ \ge\ \left(\alpha_4 - \frac{\bar\lambda}{p\lambda_1}\right)\ \|u\|_W^{p} + \frac{\alpha_4}{(s+1)^p} \||u|^s u\|_W^{p}
  - b_1 |u|_q^{q}
\quad \hbox{for all $u \in X$,}
\]
with $\alpha_4 - \frac{\bar\lambda}{p\lambda_1} > 0$.
Now, since \eqref{cps0} holds, by Sobolev Embedding Theorem and \eqref{qq0}, we have
\[
\int_\Omega |u|^{q} dx\ =\ \int_\Omega ||u|^su|^{\frac{q}{s+1}} dx\ \le\ b_2 \||u|^su\|_{W}^{\frac{q}{s+1}}\ \le\ 
b_2 \big[\ell_{W,s}(u)\big]^{\frac{q}{s+1}}
\]
for some $b_2 > 0$. Thus, from the previous estimates it follows that there exist 
$b_3$, $b_4>0$ such that
\[
\J(u)\ \ge\ b_3\ \big[\ell_{W,s}(u)\big]^{p} - b_4 \big[\ell_{W,s}(u)\big]^{\frac{q}{s+1}}
\quad \hbox{for all $u \in X$.}
\]
Whence, from \eqref{ccps} some strictly positive constants $r_0$, $\varrho_0 > 0$ can be chosen so that
$\J(u) \ge \varrho_0$ if $\ell_{W,s}(u) = r_0$.\\
On the other hand, taking any $v^* \in X\setminus \{0\}$, by Proposition \ref{geo2} with $V = {\rm span}\{v^*\}$
and the equivalence of $\|\cdot\|_X$ and $\|\cdot\|_W$ in $V$,
an element $e\in V$ exists such that $\|e\|_W > r_0$ and $\J(e) \le 0$.\\
Whence, as without loss of generality we can assume $\int_\Omega A(x,0,0) dx = 0$, 
it is $\J(0) = 0$, so Proposition \ref{wcps} and Theorem \ref{mountainpass}
imply that $\J$ has at least a nontrivial critical point. 
\end{proof}
\bigskip

\begin{proof}[Proof of Theorem \ref{mainmolti}.]
For simplicity, if $r > 0$ we set
\[
\M_r = \{u \in X:\; \ell_{W,s}(u)= r\}. 
\]
We note that $\M_r$ is the boundary of a neighborhood of the origin which is 
symmetric and bounded with respect to $\|\cdot\|_W$.\\
Then, fixing any $\varrho > 0$, from Proposition \ref{geo1} an integer $n \in \N$ and
a constant $r_{n} > 0$ exist such that \eqref{qq} holds, i.e.
\[
u \in \M_{r_{n}} \cap W_{n}^X \quad \then\quad \J(u) \ge \varrho.
\]
Now, taking any $m > n$, from \eqref{decompongo2}  
the $m$--dimensional space $V_m$ is such that 
$\codim W^X_{n} < \dim V_m$; thus, Proposition \ref{geo2} and the previous remarks
imply that assumption $({\cal H}_\varrho)$ in Theorem \ref{abstract} holds.\\
At last, without loss of generality we can assume $\int_\Omega A(x,0,0) dx = 0$, 
then $\J(0) = 0$ and for the arbitrariness of $\varrho > 0$ and 
Proposition \ref{wcps} we have that Corollary \ref{multiple} applies.
\end{proof}

\begin{proof}[Proof of Theorem \ref{mod0}.]
The proof follows from Theorem \ref{mainmolti} and Remark \ref{caso0} with $g(x,t) = |t|^{\mu-2}t$ 
and so $q = \mu$.
\end{proof}



\begin{thebibliography}{99}

\bibitem{AR} 
A. Ambrosetti and P.H. Rabinowitz, 
Dual variational methods in critical point theory and applications, 
{\em J. Funct. Anal.} {\bf 14} (1973), 349-381.

\bibitem{AB1} D. Arcoya and L. Boccardo, Critical points
for multiple integrals of the calculus of variations, {\em Arch.
Rational Mech. Anal.} {\bf 134} (1996), 249-274.

\bibitem{AB2} D. Arcoya and L. Boccardo, Some remarks on critical point
theory for nondifferentiable functionals, {\em NoDEA Nonlinear
Differential Equations Appl.} {\bf 6} (1999), 79-100.

\bibitem{ABO} D. Arcoya, L. Boccardo and L. Orsina, 
Critical points for functionals with quasilinear singular Euler--Lagrange 
equations, \emph{Calc. Var. Partial Differential Equations} {\bf 47} (2013), 159-180.

\bibitem{BBF}
P. Bartolo, V. Benci and D. Fortunato,
{Abstract critical point theorems and applications 
to some nonlinear problems with ``strong'' resonance at infinity},
\emph{Nonlinear Anal.} {\bf 7} (1983), 981-1012.

\bibitem{CP1} A.M. Candela and G. Palmieri, {Multiple solutions
of some nonlinear variational problems}, \emph{Adv. Nonlinear Stud.} {\bf 6} (2006), 269-286.

\bibitem{CP2}
A.M. Candela and G. Palmieri,
{Infinitely many solutions of some nonlinear variational equations},
\emph{Calc. Var. Partial Differential Equations} {\bf 34} (2009), 495-530.

\bibitem{CP3} A.M. Candela and G. Palmieri, {Some abstract critical point theorems
and applications}. In: \emph{Dynamical Systems, Differential Equations and Applications} 
(X. Hou, X. Lu, A. Miranville, J. Su \& J. Zhu Eds), 
\emph{Discrete Contin. Dynam. Syst.} \textbf{Suppl. 2009} (2009), 133-142. 

\bibitem{CP2017}
A.M. Candela and G. Palmieri,
Multiplicity results for some nonlinear elliptic problems
with asymptotically $p$--linear terms,
\emph{Calc. Var. Partial Differential Equations} \textbf{56}:72 (2017).

\bibitem{CS2018}
A.M. Candela and A. Salvatore, 
Positive solutions for some generalized $p$--Laplacian type problems,
\emph{Discrete Contin. Dyn. Syst. Ser. S} (to appear) .

\bibitem{Ca} A. Canino, Multiplicity of solutions for quasilinear elliptic equations,
{\em Topol. Methods Nonlinear Anal.} {\bf 6} (1995), 357-370.

\bibitem{LU} O.A. Ladyzhenskaya and N.N. Ural'tseva, \emph{Linear and Quasilinear Elliptic
Equations}, Academic Press, New York, 1968.

\bibitem{Lin} P. Lindqvist, On the equation 
${\rm div} (|\nabla u|^{p-2}\nabla u) + \lambda |u|^{p-2}u =0$, {\em Proc. Amer. Math. Soc.}
{\bf 109} (1990), 157-164.

\bibitem{LWW}
J.Q. Liu, Y.Q. Wang and Z.Q. Wang, Soliton solutions for quasilinear Schr\"odinger equations, II,
{\em J. Differential Equations} \textbf{187} (2003), 473-493.

\bibitem{PeSq} B. Pellacci and M. Squassina, Unbounded critical points for a class
of lower semicontinuous functionals, {\em J. Differential Equations} {\bf 201}
(2004), 25-62.

\bibitem{PAO}
K. Perera, R.P. Agarwal and D. O'Regan,
\emph{Morse Theoretic Aspects of $p$--Laplacian Type Operators},
Math. Surveys Monogr. {\bf 161}, Amer. Math. Soc., Providence RI, 2010.

\bibitem{Ra1}
P.H. Rabinowitz,
{\sl Minimax Methods in Critical Point Theory with Applications to Differential 
Equations}, CBMS Regional Conf. Ser. in Math. {\bf 65}, Amer. Math. Soc., Providence, 1986.

\bibitem{Str} 
M. Struwe, {\em Variational Methods. Applications to Nonlinear Partial 
Differential Equations and Hamiltonian Systems,} 4rd Edition,
Ergeb. Math. Grenzgeb. (4) {\bf 34}, Springer-Verlag, Berlin, 2008. 

\end{thebibliography}
\end{document}